\documentclass[12pt, letterpaper]{article} 
\pdfoutput=1

\newcommand{\papertitleLC}{{Ranks of quotients, remainders and
      $p$-adic digits of matrices}}

\newcommand{\paperauthors}{{Mustafa Elsheikh, Andy Novocin, Mark Giesbrecht}}

\usepackage{amsfonts}
\usepackage{amsthm}
\usepackage{amssymb} 
\usepackage{amsmath} 
\usepackage[pdfauthor=\paperauthors, pdftitle=\papertitleLC]{hyperref}

\newtheorem{theorem}{Theorem}
\newtheorem{lemma}[theorem]{Lemma}
\newtheorem{corollary}[theorem]{Corollary}
\newtheorem{conjecture}[theorem]{Conjecture}
\newtheorem{fact}[theorem]{Fact}
\DeclareMathOperator{\rem}{rem}
\DeclareMathOperator{\quo}{quo}

\newcommand{\mm}{{\mathfrak{m}}}
\newcommand{\ZZ}{{\mathbb{Z}}}

\newcommand{\nxn}{{n\times n}}

\DeclareMathOperator{\rank}{rank}
\DeclareMathOperator{\diag}{diag}

\begin{document}


\title{\papertitleLC}

\author{Mustafa Elsheikh\thanks{Cheriton School of Computer Science, University
    of Waterloo, Waterloo, Ontario, Canada (melsheik@uwaterloo.ca,
    andy@novocin.com, mwg@uwaterloo.ca). Supported by the Natural Sciences and
    Engineering Research Council (NSERC) of Canada.} \and Andy
  Novocin\footnotemark[1] \and Mark Giesbrecht\footnotemark[1]}


\pagestyle{myheadings}
\markboth{M. Elsheikh, A. Novocin, and M. Giesbrecht}{\papertitleLC}

\date{}

\maketitle


\begin{abstract}
  For a prime $p$ and a matrix $A \in \ZZ^{n \times n}$, write $A$ as $A = p (A
  \quo p) + (A \rem p)$ where the remainder and quotient operations are applied
  element-wise. Write the $p$-adic expansion of $A$ as $A = A^{[0]} + p A^{[1]}
  + p^2 A^{[2]} + \cdots$ where each $A^{[i]} \in \ZZ^{n \times n}$ has entries
  between $[0, p-1]$.  Upper bounds are proven for the $\ZZ$-ranks of $A \rem
  p$, and $A \quo p$. Also, upper bounds are proven for the $\ZZ/p\ZZ$-rank of
  $A^{[i]}$ for all $i \ge 0$ when $p = 2$, and a conjecture is presented for
  odd primes.
\end{abstract}





\section*{Outline}

This paper presents two related results on integer matrices after
applying
 element-wise 
division with remainder.
First, let $A$ be an $n \times n$
integer matrix with rank $r$ over $\ZZ$ and rank $r_0$ over $\ZZ/p\ZZ$. If $n >
p^{r_0}$ then Theorem~\ref{thm:rank-quo-p} in Section~\ref{sec:quop} shows that
$\rank(A \rem p) \le (p^{r_0} - 1)(p + 1)/(2(p-1))$ and $\rank(A \quo p) \le r +
(p^{r_0} - 1)(p + 1)/(2(p-1))$.

The second result is concerned with the $\ZZ/p\ZZ$-ranks of $p$-adic digits of an
integer matrix. Let $U, S, V \in \ZZ^{\nxn}$ such that $U, V$ have entries from
$\{0, 1\}$, $\det U \det V \not\equiv 0 \pmod{2}$, $S = \diag(1, \ldots, 1, 0,
\ldots, 0)$, $r$ be the rank of $S$ over $\ZZ/2\ZZ$, and $n \ge 2^r$. If $M =
USV \in \ZZ^{n\times n}$, then Theorem~\ref{thm:padic-general} in
Section~\ref{sec:padic} shows that rank of $M^{[i]}$ over $\ZZ/2\ZZ$ is ${r
  \choose 2^i}$ for all $i \ge 1$.  A conjecture is presented in
Section~\ref{sec:conjecture} for the same setup, but for $p$ an odd prime.

A result on integer rank of Latin squares is also obtained.  Let $A$ be the
integer matrix of rank one formed by the outer product between the vector $(1, 2,
\ldots, p-1)$ and its transpose. Then $A \rem p$ is a Latin square on the
symbols $\{1, \ldots, p-1 \}$. It is shown in Corollary~\ref{cor:rank-latin} in
Section~\ref{sec:rank-latin} that the integer rank of this Latin square is $(p +
1)/2$.


\section{Quotient and Remainder Matrices}
\label{sec:quop}

For any integer $n$ and any prime $p$, let $n \rem p$ and $n \quo p$ denote the
(non-negative) remainder and quotient in the Euclidean division $n = qp + r$
where $0 \le r < p$. The operators $\rem p$ and $\quo p$ are naturally extended
to vectors and matrices using element-wise application.

Throughout, we utilize the notion of Smith normal form of an integer matrix.
For any matrix $A \in \ZZ^{n \times n}$ of rank $r$, there exist unimodular
matrices $U, V \in \ZZ^{n \times n}$ and a unique $n \times n$ integer matrix $S
= \diag(s_1, s_2, \ldots, s_n)$ such that $A = USV$. Furthermore, $s_i \mid
s_{i+1}$ for all $1 \le i \le n$ and $s_i = 0$ for all $r < i \le n$. $S$ is
called the Smith normal form of $A$.  For a discussion on existence and
uniqueness of Smith normal form, we refer to the reader to the textbook by
Newman~\cite{Newman:1972}. We use two notions of ranks. The integer rank of $A
\in \ZZ^{n \times n}$ is denoted by $\rank(A)$. The rank of the image of $A$ in
the finite field $\ZZ/p\ZZ$
is denoted by $\rank_{p}(A)$.  Alternatively, if $r = \rank(A)$ and the Smith
form of $A$ is $S = \diag(s_1, \ldots, s_r, 0, \ldots, 0)$, then $\rank_{p}(A) =
r_0$ is the maximal index $i$ such that $p \mid s_i$.

Finally, we use the notation $A_{*, j}$ for the $j$th column of $A \in \ZZ^{n
  \times n}$ and $a_{i, j}$ for the entry $(i, j)$ of $A$.
 
\subsection{Rank Theorem}

The following theorem is the main result of Section \ref{sec:quop}.

\begin{theorem}\label{thm:rank-quo-p}
  Let $A$ be an $n \times n$ matrix over $\ZZ$, $r = \rank(A)$, $r_0 =
  \rank_{p}(A)$, and assume $n > p^{r_0}$. Then
  \begin{enumerate}
  \item[(i)] $ \rank(A \rem p) \le (p^{r_0} - 1)(p + 1)/(2(p-1)) $.
  \item[(ii)] $ \rank(A \quo p) \le r + (p^{r_0} - 1)(p + 1)/(2(p-1)) $.
  \end{enumerate}
\end{theorem}
\begin{proof}
  We will prove part (i) in Lemma \ref{lem:rank-A-rem-p}. For part (ii), we have $A =
  (A \rem p) + p (A \quo p)$, or $p(A \quo p) = A - (A \rem p)$. For matrices $X
  = Y + Z$, $\rank$ is sub-additive and $\rank(X) \le \rank(Y) +
  \rank(Z)$. Scaling a matrix by $p$ or $-1$ does not change its rank. So
  $\rank(A \quo p) \le \rank(A) + \rank(A \rem p) = r + \rank(A \rem p)$.
\end{proof}

\begin{lemma}\label{lem:rank-A-rem-p}
  $\rank(A \rem p) \le (p^{r_0} - 1)(p + 1)/(2(p-1))$.
\end{lemma}
\begin{proof}
  Let $A = USV$ be the Smith normal form of $A$, with $S = S_r + p S_q$ where
  $S_q = S \quo p$ and $S_r = S \rem p$. Then
  \begin{equation}
    A \rem p = U S V \rem p = (U S_r V + p U S_q V) \rem p = U S_r V \rem p.
  \end{equation}  If $r_0 = \rank_p(A)$ then $S_r =
  \diag(\sigma_1, \ldots, \sigma_{r_0}, 0, \ldots, 0)$ where $\sigma_i \in [1,
  p-1]$ for all $1 \le i \le r_0$. The $j$th column of $A \rem p$ is
  \begin{equation}\label{eq:cols-of-A-rem-p}
    A_{*,j} \rem p
    = \left( \sum_{\ell = 1}^{r_0} \sigma_\ell v_{\ell, j} U_{*,\ell} \right) \rem p
    = \left( \sum_{\ell = 1}^{r_0} c_{\ell, j} U_{*, \ell} \right) \rem p,
  \end{equation}
  where $c_{\ell, j} \in [0, p-1]$. If we only consider the non-zero
  coefficients $c_{\ell, j}$, then the right-hand side of \eqref{eq:cols-of-A-rem-p} is an
  $i$-term sum $(c_{\ell_1, j} U_{*, \ell_1} + \ldots + c_{\ell_i, j} U_{*,
    \ell_i}) \rem p$, where $1 \le i \le r_0$ and $1 \le \ell_1 < \ell_2 <
  \ldots < \ell_i \le r_0$.  The coefficients $c_{\ell_k, j}$ are elements in
  $[1, p-1]$ which are units modulo $p$. In particular, we can factor
  $c_{\ell_1, j}$ from the sum, and re-write \eqref{eq:cols-of-A-rem-p} as:
  \begin{equation}\label{eq:cols-of-A-rem-p-2}
    A_{*, j} \rem p = (c_{\ell_1, j} (U_{*, \ell_1} + \alpha_{\ell_2, j}
    U_{\ell_2, j} + \ldots + \alpha_{\ell_i, j} U_{*, \ell_i})) \rem p,
  \end{equation}
  where $\alpha_{\ell_k, j} \in [1, p-1]$ for all $k$.

  Fix some $i, j$ and some non-zero assignment of $\alpha_{\ell_2, j}, \ldots,
  \alpha_{\ell_i, j}$ in \eqref{eq:cols-of-A-rem-p-2} and let $\widehat{u} =
  U_{*, \ell_1} + \alpha_{\ell_2, j} U_{\ell_2, j} \ldots + \alpha_{\ell_i, j}
  U_{*, \ell_i}$. Then \eqref{eq:cols-of-A-rem-p-2} becomes $A_{*, j} \rem p =
  (c_{\ell_1, j}\widehat{u}) \rem p$. There are $p - 1$ possible values for
  $c_{\ell_1, j}$ and hence the possible values of $A_{*, j} \rem p$ are:
  \begin{equation}\label{eq:sets-of-u-sums}
    \{ \widehat{u} \rem p, (2\widehat{u}) \rem p,
    ((p-1)\widehat{u}) \rem p \}.
  \end{equation}
  We are interested in getting an upper bound on the rank of this set of
  vectors. First note that $(x y) \rem p = (x \rem p)(y \rem p) \rem p$. So $(i
  \widehat{u}) \rem p = (i (\widehat{u} \rem p)) \rem p$ for $i \in [1, p-1]$.
  Hence the maximal rank one can achieve from \eqref{eq:sets-of-u-sums} occurs
  when (up to permutation) $\widehat{u} \rem p = (0, 1, 2, \ldots, p-1,
  \ldots)$. The rest of the entries are duplicates from the same range $[0,
  p-1]$ by the pigeonhole principle. Now apply Lemma~\ref{lem:rank-1-rem}
  to conclude that the vectors in
  \eqref{eq:sets-of-u-sums} have rank at most $(p+1) / 2$.

  Thus for each $i, j$ and non-zero assignment of $\alpha_{\ell_2, j}, \ldots,
  \alpha_{\ell_i, j}$, there are at most $(p+1)/2$ linearly independent columns
  of $A \rem p$. We now count the maximal possible number of distinct $A_{*,
    j}$'s.  There are $\binom{r_0}{i}$ possible ways to select $i$ different
  columns from the first $r_0$ columns of $U$. For each choice, there are $i-1$
  coefficients: $\alpha_{\ell_2, j}, \ldots, \alpha_{\ell_i, j}$, and
  $(p-1)^{i-1}$ possible ways to assign their non-zero values from $[1,
  p-1]$. Each choice gives a set of vectors as in \eqref{eq:sets-of-u-sums}
  whose rank is at most $(p + 1)/2$.  Summing over all $i \in [1, r_0]$, the
  maximal possible rank from the span of columns in~\eqref{eq:cols-of-A-rem-p}
  is 
  \begin{equation} 
    \sum_{i = 1}^{r_0} \binom{r_0}{i} (p-i)^{i-1}
    \frac{p+1}{2} = \frac{p^{r_0} - 1}{p-1}\frac{p+1}{2},
  \end{equation}
  using the binomial theorem.
\end{proof}

\subsection{Remainder of Rank-1 Matrices}

In this section we prove the following auxiliary result.

\begin{lemma}\label{lem:rank-1-rem}
  Let $p$ be any odd prime, $n \ge p$. Let $u \in \ZZ^{n}$ be any non-zero
  vector where the entries of $u \rem p$ include $\{ 1, 2, \ldots, p-1 \}$.
  Then the set of vectors $\{ u \rem p, (2u) \rem p, \ldots, ((p-1)u) \rem p \}$
  is linearly dependent and has rank $(p+1)/2$.
\end{lemma}

First we prove this result for $n = p-1$. A generalization follows. Let $u = (1,
2, \ldots, p-1) \in \ZZ^{(p-1)}$ and $M \in \mathbb{Z}^{(p-1)\times (p-1)}$ be
the rank-$1$ matrix $M = u u^{T}$ and $R = M \rem p$.

\begin{lemma}\label{lem:rank-R}
  $\rank(R) = (p+1)/2$.
\end{lemma}
\begin{proof}
  Lemma \ref{lem:rank-R-upper} shows that $(p+1)/2$ is an upper bound on the
  rank and Lemma~\ref{lem:rank-R-lower} shows that $(p+1)/2$ is a lower
  bound.
\end{proof}


\begin{lemma}\label{lem:rank-R-upper}
  $\rank(R) \le (p+1)/2$.
\end{lemma}
\begin{proof}
  Let $1 \le j \le (p-1)/2$ and $1 \le i \le p-1$. Write $ij = qp + r$ where $0
  \le r < p$. Also $i,j < p \implies p \nmid i \land p \nmid j$, which implies
  $r \neq 0$. Then $i(p-j) = (i-q-1) + (p-r)$ where $0 < (p-r) < p$. So $ij \rem
  p + i(p-j) \rem p = r + (p-r) = p$. But $R_{i,j} = ij \rem p$, so for all $1
  \le i \le (p-1)/2$ we have $R_{*,i} = (p, p, \ldots, p)^{T} - R_{*,
    p-i}$. Thus there are $(p-1)/2$ linearly dependent columns, and no more than
  $(p+1)/2$ linearly independent columns.
\end{proof}





To prove that $(p+1)/2$ is also a lower bound on the rank, it suffices (using
Lemma~\ref{lem:rank-R-upper}) to consider the matrix $B$ of size ${(p-1) \times
  \frac{p+1}{2}}$ which is formed by the first $(p-1)/2$ columns of $R$ and the
column $B_{*, (p+1)/2} = R_{*, (p+1)/2} + R_{*, (p-1)/2} = (p, \ldots, p)^{T}$.
The matrix $B $ has the following structure:
\[ B = \begin{bmatrix}
  1 & 2 & \cdots & \frac{p-1}{2} & p \\
  2 & 4 & \cdots & p-1 & p \\
  3 & 6 \rem p & \cdots & 3\frac{p-1}{2} \rem p & p \\
  \vdots & \vdots & \ddots & \vdots \\
  (p-1) \rem p & 2(p-1) \rem p & \cdots & \frac{(p-1)^2}{2}\rem p & p
  \end{bmatrix}. \]

  \begin{lemma}\label{lem:extra-induction}
  Either the right kernel of $B$ is empty, or the first $(p-1)/2$ columns of $B$
  are linearly dependent.
\end{lemma}
\begin{proof}
  If the right kernel of $B$ is not empty, then there exists $(p+1)/2$ integers
  $c_1, \ldots, c_{(p+1)/2}$ not identically zero, such that
  \begin{equation}
    c_1 B_{*, 1} + c_2 B_{*, 2} + \ldots + c_{(p+1)/2} B_{*, (p+1)/2} = 0.
  \end{equation}
  Apply this linear combination simultaneously to the first two rows of $B$ to
  get
  \begin{eqnarray}
    c_1 + 2c_2 + \ldots + c_{(p-1)/2}\ (p-1)/2 = - c_{(p+1)/2}\ p
    \label{eq:lin-comb-2} \\
    2c_1 + 4c_2 + \ldots + c_{(p-1)/2}\ (p-1) = - c_{(p+1)/2}\ p
    \label{eq:lin-comb-3}
  \end{eqnarray}
  But \eqref{eq:lin-comb-2} implies either a contradiction in
  \eqref{eq:lin-comb-3}: the right kernel of $B$ is empty, or $c_{(p+1)/2}
  = 0$ and the first $(p-1)/2$ columns of $B$ are linearly dependent.
\end{proof}

\begin{lemma}\label{lem:rank-R-lower}
  $(p+1)/2 \le \rank(R)$.
\end{lemma}
\begin{proof}
  Using Lemma \ref{lem:extra-induction}, proving a lower bound on the rank of
  $R$ can be reduced to showing that the first $(p-1)/2$ columns of $B$ are
  linearly independent. We use induction. Consider the sequence of matrices
  $B^{(k)}$ formed by the first $k$ columns of $B$, where $2 \le k \le (p-1)/2$.
  The base case of induction, $B^{(2)}$, has rank $2$ which is straightforward
  to verify. For the inductive case, we assume $B^{(k-1)}$ has rank $k-1$, and
  use Lemma \ref{lem:inductive} to deduce that $B^{(k)}$ has rank $k$.
\end{proof} 

The following lemma is needed before proving Lemma \ref{lem:inductive}.

\begin{lemma}\label{lem:aux-3j-rem}
  For all $j \ge 1$, $(3j \rem p) - 3j = -pq$ for some integer $q \ge 0$.
\end{lemma}
\begin{proof}
  Write $3j$ as $3j = qp + r$ where $r = 3j \rem p$ and $q = 3j \quo p$. Then $r
  - 3j = -qp$.
\end{proof}

\begin{lemma}\label{lem:inductive}
  Let $B^{(k)}$ be the $(p-1) \times k$ integer matrix in proof of
  Lemma~\ref{lem:rank-R-lower}. Either $B^{(k)}$ has column rank $k$, or
  $B^{(k-1)}$ is rank deficient.
\end{lemma}
\begin{proof}
  If the right kernel of $B^{(k)}$ was not empty, then there exists integers
  $c_1, \ldots, c_k$ not identically zero, such that
  \begin{eqnarray}\label{eq:lin-comb-matrix}
    \begin{bmatrix}
       1 & 2 & \cdots & k \\
       2 & 4 & \cdots & 2k \\
       3 & 6 \rem p & \cdots & (3k) \rem p \\
       & & \ddots
    \end{bmatrix}
    \begin{bmatrix} c_1 \\ \vdots \\ c_k \end{bmatrix}
    =
    \begin{bmatrix} 0 \\ \vdots \\ 0 \end{bmatrix}.
  \end{eqnarray}
  We then perform the following row operations on the left-hand side of
  \eqref{eq:lin-comb-matrix}: replace (row $3$) by (row 3$)- 3 \times ($row
  $1$), then divide row $3$ by $-p$. From Lemma \ref{lem:aux-3j-rem}, we have
  that row $3$ is now
  \[ \begin{bmatrix} 0 & \cdots & 0 & 1 & \cdots & 1 & 2 \cdots &
    q \end{bmatrix}, \]
  for some $q$ (in fact, $q = (3k) \quo p$). We then
  perform the following column operations: let $\ell$ denote the column index
  where the first $1$ appears in row $3$ ($\ell$ is guaranteed to be greater
  than or equal $1$ since for all $p > 3$, $k \le (p-1)/2$, we have $3k > p$.)
  Pivot on entry $\ell$ in row $3$ and eliminate all entries of row $3$ with
  indices between $\ell + 1$ and $k - 1$. Subtract $q-1$ multiples of column
  $\ell$ from column $k$. Then pivot on entry $k$ of row $3$ and subtract column
  $k$ from column $\ell$. Effectively, this sequence of operations transforms
  row $3$ into:
  \[ \begin{bmatrix} 0 & \cdots & 0 & 1 \end{bmatrix}. \]
  The right-hand side
  of \eqref{eq:lin-comb-matrix} is zero, and hence not effected by the
  aforementioned elementary operations.
  
  Finally, the transformed row $3$ implies either that $c_k$ is zero, or the
  existence of $c_1, \ldots, c_k$ is contradictory. This proves the statement of
  the lemma.
\end{proof}


We are now ready to generalize Lemma \ref{lem:rank-R} and prove Lemma
\ref{lem:rank-1-rem}.
\begin{proof}[Proof of Lemma~\ref{lem:rank-1-rem}]
  For the column vector $u \in \ZZ^{n\times 1}$, consider the matrix
  $\widehat{R} \in \ZZ^{n\times n} = u u^{T} \rem p$, which is analogous to the
  matrix $R$ of Lemma \ref{lem:rank-R}.  The image of $u \rem p$ has entries
  from the interval $[0, p-1]$. If $n > p$ then, by the pigeonhole principle,
  the vector $u \rem p$ will contain duplicate (and zero) entries, which
  correspond to duplicate and zero rows in $\widehat{R}$. So up to row/column
  permutations, $\widehat{R}$ contains $R$ as a submatrix, and the extra
  rows/columns are duplicate and/or zero. Hence $\rank(\widehat{R}) = \rank(R)$.
\end{proof}

\subsection{A Note on Ranks of Latin Squares}\label{sec:rank-latin}

It is worth noting that Lemma~\ref{lem:rank-R} also implies a result on the
ranks of Latin squares of certain orders. As before, let $p$ be an odd prime, and
let $R$ be the $(p-1) \times (p-1)$ integer matrix whose $(i, j)$th entry is $ij
\rem p$. We show that $R$ is a Latin square as follows. $R$ is the Cayley
multiplication table of the finite field $\ZZ/p\ZZ$, excluding the element $0$.
Since $\ZZ/p\ZZ$ is an integral domain, we have $ij \rem p \neq
ij' \rem p$ whenever $j \neq j'$ (where $i, j, j' \in [1, p-1]$).
So every row/column of $R$ has the residues
$\{ 1, \ldots, p-1 \}$ appearing only once, and $R$ is a Latin square of order
$p-1$. $R$ has rank $1$ over $\ZZ/p\ZZ$ and non-trivial rank over $\ZZ$ by
Lemma~\ref{lem:rank-R} as stated in the following corollary.

\begin{corollary}\label{cor:rank-latin}
  Let $p$ be any odd prime, and let $R$ be any Latin square of order $p-1$ on
  the symbols $\{ 1, \ldots, p-1 \}$. Then the integer rank of $R$, taken as a
  $(p - 1) \times (p - 1)$ integer matrix, is $(p + 1)/2$.
\end{corollary}



\section{$p$-adic Matrices}\label{sec:padic}

We now switch the focus to ranks of $p$-adic matrices. Ranks in this section are
over the finite field with $p$ elements\footnote{The two ranks, over $\ZZ$ and
  over $\ZZ/p\ZZ$, are equal unless $p$ is an elementary divisor of the
  matrix.}, with residue classes $\{ 0, 1, \ldots, p - 1 \}$.
For any prime $p$ and any matrix $M \in \ZZ^{n\times n}$ with entries $| m_{i,
  j} | < \beta$, the $p$-adic expansion of $M$ is $M = M^{[0]} + p M^{[1]} +
\ldots + p^s M^{[s]}$ where the entries of each matrix $M^{[i]}$ are between
$[0, p-1]$, and $s \le \lceil \log_p \beta \rceil$. We call $M^{[i]}$ the $i$th
$p$-adic matrix digit of $M$. We extend the superscript ${[i]}$ notation to
vectors and integers in the obvious way.

We present results concerning the $2$-adic matrix digits. For odd primes, we
only present a conjecture. It is an open question to study the combinatorial
structure of the column space of the $p$-adic matrix digits for primes other
than $2$.

\subsection{Binary code matrices}

Fix $p = 2$. The goal of this section is to show that for all $i \ge 1$,
$\rank_p(M^{[i]}) = {r \choose 2^i}$ where $M = A A^{T}$ for some specially
constructed $A$, which we call \emph{binary code} matrix. We will generalize the
construction of $M$ in a subsequent section. For now, $A$ is constructed as
follows. Start with the $2^r \times r$ matrix whose ${i, j}$ entry is the $j$th
bit in the binary expansion of $i$.  Then apply row permutations to $A$ such
that the first $r \choose 0$ rows have have exactly $0$ non-zero entries,
followed by $r \choose 1$ rows which have exactly $1$ non-zero entries, followed
by $r \choose 2$ rows which have exactly $2$ non-zero entries and so on. See
Figure \ref{fig:binarycode-example} for an example where $r = 4$.

\begin{figure*}[t]
  \label{fig:binarycode-example}
  \begin{equation*}
    \left[ \begin{array}{cccc}
        0 & 0 & 0 & 0 \\
        \hline
        0 & 0 & 0 & 1 \\
        0 & 0 & 1 & 0 \\
        0 & 1 & 0 & 0 \\
        1 & 0 & 0 & 0 \\
        \hline
        0 & 0 & 1 & 1 \\
        0 & 1 & 0 & 1 \\
        0 & 1 & 1 & 0 \\
        1 & 0 & 0 & 1 \\
        1 & 0 & 1 & 0 \\
        1 & 1 & 0 & 0 \\
        \hline
        0 & 1 & 1 & 1 \\
        1 & 0 & 1 & 1 \\
        1 & 1 & 0 & 1 \\
        1 & 1 & 1 & 0 \\
        \hline
        1 & 1 & 1 & 1
      \end{array} \right],
    \left[ \begin{array}{c | cccc | cccccc| cccc | c}
        0 & 0 & 0 & 0 & 0 & 0 & 0 & 0 & 0 & 0 & 0 & 0 & 0 & 0 & 0 & 0 \\
        \hline
        0 & 1 & 0 & 0 & 0 & 1 & 1 & 0 & 1 & 0 & 0 & 1 & 1 & 1 & 0 & 1 \\
        0 & 0 & 1 & 0 & 0 & 1 & 0 & 1 & 0 & 1 & 0 & 1 & 1 & 0 & 1 & 1 \\
        0 & 0 & 0 & 1 & 0 & 0 & 1 & 1 & 0 & 0 & 1 & 1 & 0 & 1 & 1 & 1 \\
        0 & 0 & 0 & 0 & 1 & 0 & 0 & 0 & 1 & 1 & 1 & 0 & 1 & 1 & 1 & 1 \\
        \hline
        0 & 1 & 1 & 0 & 0 & 2 & 1 & 1 & 1 & 1 & 0 & 2 & 2 & 1 & 1 & 2 \\
        0 & 1 & 0 & 1 & 0 & 1 & 2 & 1 & 1 & 0 & 1 & 2 & 1 & 2 & 1 & 2 \\
        0 & 0 & 1 & 1 & 0 & 1 & 1 & 2 & 0 & 1 & 1 & 2 & 1 & 1 & 2 & 2 \\
        0 & 1 & 0 & 0 & 1 & 1 & 1 & 0 & 2 & 1 & 1 & 1 & 2 & 2 & 1 & 2 \\
        0 & 0 & 1 & 0 & 1 & 1 & 0 & 1 & 1 & 2 & 1 & 1 & 2 & 1 & 2 & 2 \\
        0 & 0 & 0 & 1 & 1 & 0 & 1 & 1 & 1 & 1 & 2 & 1 & 1 & 2 & 2 & 2 \\
        \hline
        0 & 1 & 1 & 1 & 0 & 2 & 2 & 2 & 1 & 1 & 1 & 3 & 2 & 2 & 2 & 3 \\
        0 & 1 & 1 & 0 & 1 & 2 & 1 & 1 & 2 & 2 & 1 & 2 & 3 & 2 & 2 & 3 \\
        0 & 1 & 0 & 1 & 1 & 1 & 2 & 1 & 2 & 1 & 2 & 2 & 2 & 3 & 2 & 3 \\
        0 & 0 & 1 & 1 & 1 & 1 & 1 & 2 & 1 & 2 & 2 & 2 & 2 & 2 & 3 & 3 \\
        \hline
        0 & 1 & 1 & 1 & 1 & 2 & 2 & 2 & 2 & 2 & 2 & 3 & 3 & 3 & 3 & 4
      \end{array} \right]
  \end{equation*}
  \caption{An example of $A$ (left) and $M = AA^T$ (right), where $r = 4$. The
    rows of $A$ are partitioned by the number of non-zero entries in each row.
    The corresponding blocks in the symmetric matrix $M$ are shown with
    borders. The column partitions of $M$ are $\mm_0$, $\mm_1$, $\mm_2$,
    $\mm_3$, $\mm_4$.  And $\rank_p(M^{[0]}) = \rank_p(\mm_1^{[0]}) = 4$,
    $\rank_p(M^{[1]}) = \rank_p(\mm_2^{[1]}) = 6$, $\rank_p(M^{[2]}) =
    \rank_p(\mm_4^{[2]}) = 1$.}
\end{figure*}


The $\ell$th column of $M$ is given by:
\begin{equation}\label{eq:col-sum}
  M_{*,\ell} = a_{1, \ell} A_{*, 1} + \ldots + a_{r, \ell} A_{*, r} = \sum_{j
    \in J_\ell} A_{*,j},
\end{equation}
where $J_{\ell} \subseteq \{ 1, 2, \dots, r \}$ and the second equality holds
because $a_{i, \ell} \in \{ 0, 1 \}$.  We call $J_\ell$ the \emph{summing index
  set} of $M_{*, \ell}$. Let $\mm_k$ denote the $2^r \times {r \choose k}$
submatrix of $M$, which includes all columns of the form: $M_{*, \ell} = \sum_{j
  \in J_{\ell}} A_{*, j}$ where $J_{\ell} \subseteq \{ 1, 2, \ldots, r \}$ and
$| J_\ell | = k$. Then the columns of $M$ can be partitioned into:
\begin{equation}
  M = \begin{bmatrix} \mm_0 & \mm_1 & \mm_2 & \ldots & \mm_{2^i} & \mm_{2^i + 1} &
    \ldots & \mm_{r} \end{bmatrix}.
\end{equation}
The next lemma shows that
\begin{equation}
  M^{[i]} = \begin{bmatrix} \mathbf{0} & \mathbf{0} & \ldots & \mathbf{0} &
    \mm_{2^i}^{[i]} & \mm_{2^i + 1}^{[i]} & \ldots &
    \mm_{r}^{[i]} \end{bmatrix}.
\end{equation}

\begin{lemma}\label{lem:at-least-2toi}
  If $k < 2^i$, then $\mm_k^{[i]} = \mathbf{0}$ for all $i \ge 1$.
\end{lemma}
\begin{proof}
  Columns of $\mm_k$ are given by $\sum_{j \in J} A_{*,j}$ where $|J| = k$.  The
  entries of $A$ are either $0$ or $1$. So the largest entry in $\mm_k$ is $1 +
  \ldots + 1 = k$. The result follows by appealing to the binary expansion of
  $k$.
\end{proof}

We expect $\rank_p(\mm_{2^i}^{[i]}) \le {r \choose 2^i}$ since $\mm_{2^i}^{[i]}$
is a matrix of dimension $2^r \times {r \choose 2^i}$. The next lemma shows that
the rank is, in fact, equal to this upper bound.

\begin{lemma}\label{lem:m-2i-rank-equal}
  $\rank_p(\mm_{2^i}^{[i]}) = {r \choose 2^i}$ for all $i \ge 1$.
\end{lemma}
\begin{proof}
  Let $c_1, \ldots, c_{r \choose 2^i}$ be the column indices of $\mm_{2^i}$ in
  $M$.  Let $S(\mm_{2^i})$ be the ${r \choose 2^i} \times {r \choose 2^i}$
  submatrix of $\mm_{2^i}$ formed by the rows $c_1, \ldots, c_{r \choose 2^i}$,
  and $S(A)$ be the ${r \choose 2^i} \times r$ submatrix of $A$ formed by the
  rows $c_1, \ldots, c_{r \choose 2^i}$. Rows of $S(A)$ have exactly $2^i$
  non-zero entries because of the construction of $A$. If we treat $A$ and $M$
  as block matrices then $S(\mm_{2^i}) = S(A) S(A)^T$ is the $2^i$th diagonal
  block of $M$ (See Figure \ref{fig:binarycode-example}).

  The entries in row $\rho$ of $S(\mm_{2^i})$ are given by linear combinations
  of the entries in row $\rho$ of $S(A)$. The summing index sets $J_j$,
  where $|J_j| = 2^i$,
  are exactly the locations of the non-zero entries of rows of $S(A)$, which are
  all \emph{different} by construction. Hence there is only one entry in row
  $\rho$ of $S(\mm_{2^i})$ whose summing set matches the locations of the
  non-zero entries in row $\rho$ of $S(A)$. The value of this entry is $1 + 1 +
  \ldots + 1 = 2^i$. The other entries have values less than $2^i$.  Now appeal
  to the binary expansion of $2^i$ to get that $S(\mm_{2^i}^{[i]})$ is an
  identity (sub)matrix\footnote{This is true in the example of Figure
    \ref{fig:binarycode-example} without any reordering, because we constructed
    the row blocks of $A$ such that the binary expansion of $i$ comes after the
    binary expansion of $j$ whenever $i > j$. Without such ordering, the
    identity block assertion holds up to row and column permutations.} of
  $\mm_{2^i}^{[i]}$ whose size is ${r \choose 2^i} \times {r \choose
    2^i}$. Therefore, $\mm_{2^i}^{[i]}$ has rank ${r \choose 2^i}$.
\end{proof}

Next we will prove that $\rank_p(M^{[i]}) = {r \choose 2^i}$ by showing that all
the columns of $ \mm_{2^i + 1}^{[i]}, \mm_{2^i + 2}^{[i]}, \ldots,
\mm_{2^r}^{[i]} $ are linearly \emph{dependent} on those of $\mm_{2^i}^{[i]}$.

\begin{lemma}
  Consider any column $m$ in $\mm_{2^i + z}$, where $z \ge 1$. Then $m^{[i]}$ is
  a linear combination of columns of $\mm_{2^i}^{[i]}$.
\end{lemma}
\begin{proof}
  Let $J$ be the summing index set of $m$, where $ | J | = 2^i + z$. Let
  $\mathcal{I}$ be the set of all subsets of $J$ of size $2^i$, so $|
  \mathcal{I} | = {2^i + z \choose 2^i}$. For every $I \in \mathcal{I}$, there
  is a unique corresponding column $c_I$ in $\mm_{2^i}$ whose summing set is
  $I$. We will show that $m^{[i]}$ can be obtained by adding up $c_I$'s. In
  other words,
  \begin{equation}\label{eq:m-is-linear-sum}
    m^{[i]} \equiv \sum_{I \in \mathcal{I}} c_I^{[i]} \pmod{2}. 
  \end{equation}
  Let $A_J$ denote the submatrix of $A$ formed by the columns indexed by
  $J$. For any row $\rho$ of $A_J$, let $2^i + k_\rho$ be the number of $1$'s in
  that row, where $-2^i \le k_\rho \le z$. First, if $k_\rho < 0$, then the
  corresponding sum of $1$'s at this row is less than $2^i$. By
  Lemma~\ref{lem:at-least-2toi}, we have the corresponding entries in both
  $\mm_{2^i}^{[i]}$ and $\mm_{2^i + z}^{[i]}$ are zeros and
  \eqref{eq:m-is-linear-sum} trivially holds.  On the other hand, if $0 \le
  k_\rho \le z$, then the $\rho$th entry of the right-hand side of
  \eqref{eq:m-is-linear-sum} is $1 + 1 + \ldots + 1 \equiv {2^i + k_\rho \choose
    2^i} \pmod{2}$ since $| \mathcal{I} | = {2^i + k_\rho \choose 2^i}$. (Recall
  that the number of non-zero entries in row $\rho$ is $2^i + k_\rho$ rather
  than $2^i + z$.) The $\rho$th entry of the left-hand side of \eqref{eq:m-is-linear-sum} is
  $(2^i + k_\rho) \quo 2^i$. The $(2^i + k_\rho)$ term corresponds to adding
  $(2^i + k_\rho)$ non-zero entries, and the $\quo 2^i$ operation corresponds to
  the $i$th bit of the binary expansion of $m$. By Lemma~\ref{lem:kummer-equiv}
  (below), we have $(2^i + k_\rho) \quo 2^i \equiv {2^i + k_\rho \choose 2^i}
  \pmod{2}$, and \eqref{eq:m-is-linear-sum} holds.
\end{proof}

The proof of the next (auxiliary) lemma uses a theorem due to
Kummer~\cite{Kummer:1851}.

\begin{fact}[Kummer's Theorem] The exact power of $p$
  dividing $a+b \choose a$ is equal to the number of carries when performing the
  addition of $(a+b)$ written in base $p$.
\end{fact}

A corollary of Kummer's theorem is that ${a+b \choose a}$ is odd (resp. even) if
adding $(a+b)$ written in binary expansion generates no (resp. some) carries.

\begin{lemma}\label{lem:kummer-equiv}
  $ (2^i + k) \quo 2^i \equiv {2^i + k \choose 2^i} \pmod 2 $.
\end{lemma}
\begin{proof}
  We will show that $(2^i + k) \quo 2^i$ and $2^i + k \choose 2^i$ have the same
  parity.
  Write $k = Q 2^i + R$ for a quotient $Q \ge 0$ and a remainder $0 \le R <
  2^i$. There are two cases for $Q$.
  If $Q$ is even, then the $i$th bit\footnote{i.e. the coefficient of $2^i$ in
    the binary expansion of $k$.} of $k$ is $0$ and hence no carries are
  generated when adding $k$ and $2^i$ in base $2$. So by Kummer's Theorem, $2^i
  + k \choose 2^i$ is odd and ${2^i + k \choose 2^i} \equiv 1 \pmod{2}$.
  If $Q$ is odd, then the $i$th bit of $k$ is $1$ and the number of carries
  generated when adding $2^i + k$ in base $2$ is at least $1$. So by Kummer's
  theorem $2^i + k \choose 2^i$ is even and ${2^i + k \choose 2^i} \equiv 0
  \pmod{2}$.

  We have shown that ${2^i + k \choose 2^i}$ and $Q$ have opposite
  parities. Now, substitute $k = Q 2^i + R$ to get $(2^i + k) \quo 2^i =
  Q+1$. Hence, modulo $2$, $(2^i + k) \quo 2^i$ also have an opposite parity to
  that of $Q$. This concludes our proof.
\end{proof}

\subsection{Non-symmetric Matrices}

So far we have shown that $\rank_p(M^{[i]}) = \rank_p(\mm_{2^i}^{[i]}) = {r \choose
  2^i}$, where $M = A A^{T}$ for some specially constructed $A$. We now put the
results together into a more general theorem.

\begin{theorem}\label{thm:padic-general}
  Assume $U, S, V \in \ZZ^{\nxn}$, such that $U, V$ have entries from $\{0,
  1\}$, $\det U \det V \not\equiv 0 \pmod{2}$, $S = \diag(1, \ldots, 1, 0,
  \ldots, 0)$, $\rank_p(S) = r$, and $n \ge 2^r$. If $M = USV \in \ZZ^{n\times
    n}$, then $\rank_p(M^{[i]}) = {r \choose 2^i}$ for all $i \ge 1$.
\end{theorem}
\begin{proof}
  Since $S = SS$, we have $M = USV = USSV = LR$, where $L = US \in \ZZ^{n\times
    r}$, and $R = SV \in \ZZ^{r\times n}$. Let $A \in \ZZ^{2^r \times r}$ be the
  binary code matrix of the digits $\{ 0, \ldots, 2^r - 1\}$. Consider the
  matrices $\widehat{L} = A$, $\widehat{R} = A^T$ and $\widehat{M} = \widehat{L}
  \widehat{R}$. If we start with $\widehat{L}$ (resp. $\widehat{R}$) and augment
  it with the appropriate $(n - 2^r)$ additional rows (resp. columns), and apply
  the appropriate row and column permutations, then we could transform
  $\widehat{L}$ into $L$ (resp. $\widehat{R}$ into $R$), and in effect,
  transform $\widehat{M}$ into $M$. Our goal is to show that the rank arguments
  of the previous lemmas hold under the aforementioned operations.

  We first note that row and column permutations preserve ranks. Also, by a
  simple enumeration argument over the binary tuples of size $r$, and by the
  given fact that $n \ge 2^r$, we can conclude that any additional rows
  (resp. columns) augmented to $\widehat{L}$ (resp. $\widehat{R}$) will be
  linearly dependent. In fact, any such rows (resp. columns) will be duplicates
  of existing rows (resp. columns).

  Now, consider adding extra columns to $\widehat{R}$. The resulting extra
  columns in $\widehat{M}$
  are duplicates of existing columns and hence the ranks in
  Lemma~\ref{lem:m-2i-rank-equal} are not affected. Finally, adding extra rows
  to $\widehat{L}$ does not change the cardinality of the summing index sets in
  \eqref{eq:col-sum}.  The rest of the results are straightforward to verify.
\end{proof}

\subsection{Odd Primes}\label{sec:conjecture}

For $p = 2$, the non-zero patterns of the binary code matrix $A$ coincides with
the summing indices in \eqref{eq:col-sum}. This is not true for odd primes,
where the linear combinations can have coefficients other than $0$ and $1$. Thus
it is an open question to devise construction a similar to binary code matrices,
which exposes the combinatorial structure of the column space of $M = A
A^T$. However, we present the following conjecture towards understanding the
$p$-adic ranks for odd primes.

\begin{conjecture}\label{RankCarryConjecture}
  Assume $p = 2k + 1$ is an odd prime, $U,S,V \in \ZZ^{n \times n}$ such that
  $U,V$ have entries from $[0, p-1]$, $\det U \det V \not\equiv 0 \pmod{p}$, $S$
  is a $0, 1$ diagonal matrix and $\rank_p(S) = r$.  Let $M = USV = M^{[0]} +
  M^{[1]} p + \cdots$ where $M^{[i]} \in (\ZZ/p\ZZ)^{n \times n}$. It is
  conjectured that
  \begin{equation}
    \rank_p(M^{[1]}) \leq \sum_{i=0}^{k}{ \binom{r+2i}{2i+1} } + \binom{r+2k - 1}{2k}
    -2r
  \end{equation}
  Furthermore, in the generic case where the entries of $U,V$ are uniformly
  chosen at random from $[0, p-1]$, and $n$ is arbitrarily large, the ranks are
  \emph{equal} to the stated bound.
\end{conjecture}

This conjecture first appeared in~\cite{Elsheikh:2012}. It shows that a product
of matrices with ``small'' entries and ``small'' rank can still have very large
rank, but not full, $p$-adic expansion. In other words, the ``carries'' from the
product $USV$ will impact many digits in the expanded product.



\section*{Acknowledgment}
The authors would like to thank Andrew Arnold, Kevin Hare, David McKinnon, 
Jason Peasgood, and B. David Saunders.



\end{document}